\documentclass[a4paper]{amsart}
\pdfoutput=1
\usepackage[latin1]{inputenc}
\usepackage{ae,aecompl,amsbsy,amssymb,amsmath,amsthm,
eurosym,amsfonts,epsfig,graphicx,graphics,verbatim,enumerate,color}
\usepackage{hyperref}
\hypersetup{colorlinks=true,linkcolor=blue}
\usepackage{tikz}
\usetikzlibrary{arrows, calc, positioning, decorations.markings}
\tikzset{
    between/.style args={#1 and #2}{
         at = ($(#1)!0.5!(#2)$)
    }
}

\let\oldaddcontentsline\addcontentsline
\newcommand{\stoptocentries}{\renewcommand{\addcontentsline}[3]{}}
\newcommand{\starttocentries}{\let\addcontentsline\oldaddcontentsline}

\newcommand{\norm}[1]{\left\| #1 \right\|}

\DeclareMathOperator*{\essinf}{ess\,inf}

\newcommand{\br}{\mathbb{R}}

\newcommand{\cm}{\mathcal{M}}
\newcommand{\cl}{\mathcal{L}}

\newcommand{\cs}{\mathcal{S}}

\DeclareMathOperator{\supp}{supp}

\newcommand{\abs}[1]{\left\lvert#1\right\rvert}

\newcommand{\dmu}{\mathrm{d}\mu}

\newcommand{\dnu}{\mathrm{d}\nu}

\usepackage{enumitem}

\theoremstyle{plain}% default
\newtheorem{theorem}{Theorem}[section]

\newtheorem{lemma}[theorem]{Lemma}
\newtheorem{proposition}[theorem]{Proposition}
\newtheorem{corollary}[theorem]{Corollary}

\theoremstyle{definition}
\newtheorem{definition}[theorem]{Definition}

\theoremstyle{remark}
\newtheorem*{remark}{Remark}
\newtheorem*{example}{Example}
\newtheorem*{example_continued}{Example (continued)}

\newtheorem{claim}{Claim}

\numberwithin{equation}{section}

\begin{document}
\date{\today}
\title[A variational proof of a disentanglement theorem]{A variational proof of a disentanglement theorem for multilinear norm inequalities}
\author{Anthony Carbery}
\author{Timo S. H\"anninen}
\author{Stef\'an Ingi Valdimarsson}
\begin{abstract}
   The {\em basic  disentanglement theorem} established by the present authors states that estimates on a weighted geometric mean over (convex) families of functions can be disentangled into quantitatively linked estimates on each family separately. On the one hand, the theorem gives a uniform approach to classical results including Maurey's factorisation theorem and Lozanovski\u{\i}'s factorisation theorem, and, on the other hand, it underpins the duality theory for multilinear norm inequalities developed in our previous two papers.

   In this paper we give a simple proof of this basic disentanglement theorem.  Whereas the approach of our previous paper was rather involved -- it relied on the use of minimax theory together with weak*-compactness arguments in the space of finitely additive measures, and an application of the Yosida--Hewitt theory of such measures -- the alternate approach of this paper is rather straightforward: it instead depends upon elementary perturbation and compactness arguments.
\end{abstract}

\address{Anthony Carbery, School of Mathematics and Maxwell Institute for Mathematical Sciences, University of Edinburgh, James Clerk Maxwell Building, Peter Guthrie Tait Rd, Kings Buildings, Edinburgh EH9 3FD, Scotland}
\email{A.Carbery@ed.ac.uk}

\address{Timo S. H\"anninen, Department of Mathematics and Statistics, University of Helsinki, P.O. Box 68, FI-00014 Helsinki, Finland}
\email{timo.s.hanninen@helsinki.fi}

\address{Stef\'an Ingi Valdimarsson,
Arion banki, Borgart\'un 19, 105 Reykjav\'ik,
Iceland; and Science
Institute, University of Iceland, Dunhagi 5, 107
Reykjav\'ik, Iceland.}
\email{sivaldimarsson@gmail.com}

\maketitle
\setcounter{tocdepth}{2}
\tableofcontents
%\subjclass[2010]{}
%     
%\keywords{}
%
\section{Introduction}

The purpose of this paper is to give a simple proof of the following {\it basic disentanglement theorem}:

\begin{theorem}[Basic disentanglement theorem, {\cite[Theorem~2.2]{chv2}}]\label{theorem:basic}
Let $(X, {\rm d}\mu)$ be a $\sigma$-finite measure space. Suppose that $\alpha_j > 0$ are such that $\sum_{j=1}^d \alpha_j = 1$. For each $j \in \{1, \dots , d\}$ let $\mathcal{G}_j$ be a saturating\footnote{A set $\mathcal{G}$ of measurable functions is called {\it saturating} if the functions vanish identically only on sets of measure zero: if $f=0$ almost everywhere on $E$ for all $f\in \mathcal{G}$, then $E$ has measure zero.} convex set of nonnegative measurable functions. Assume that 
     \begin{equation*}%\label{prz2}
    \int_X \prod_{j=1}^d g_{j}(x)^{\alpha_j} {\rm d} \mu(x)
    \leq A \quad\text{for all $g_j\in  \mathcal{G}_j$.}
\end{equation*}
 Then there exist nonnegative measurable functions $\phi_j$ such that
  \begin{equation*}%\label{prz1}
\prod_{j=1}^d \phi_j(x)^{\alpha_j} \geq 1
  \end{equation*}
almost everywhere on $X$, and such that for all $j$,
   \begin{equation*}%\label{prz2}
    \int_X g_{j}(x) \phi_j(x) {\rm d} \mu(x) \leq A\quad\text{for all $g_j\in \mathcal{G}_j$.} 
    \end{equation*}
\end{theorem}

This basic disentanglement theorem underpins the duality theory for multilinear norm inequalities developed in our previous papers \cite{chv1,chv2}. The scope of the duality theory includes Brascamp--Lieb type inequalities and multilinear Kakeya and restriction type inequalities. Each such norm inequality is equivalent to the existence of a factorisation naturally associated with it. The precise statement for positive multilinear inequalities is as follows (whose case $p_j=q=r_j=1$ is in fact equivalent to the basic disentanglement theorem):
\begin{theorem}[Duality theorem for positive multilinear inequalities, cf. Theorem~5.1 of \cite{chv2}]\label{theorem:duality}Let $\cm$ denote the space of measurable functions. For $j=1,\ldots,d$, let $1 \leq q, r_j < \infty$, 
$T_j:L^{r_j}\to \cm$ be saturating\footnote{An operator $T:L^r\to \cm$ is called {\it saturating} if it can vanish identically only on sets of measure zero: if $Tf=0$ almost everywhere on $E$ for all $f\in L^r$, then $E$ has measure zero. } positive linear operators, and $\theta_j\in(0,1)$ be weights such that $\sum_{j=1}^d\theta_j=1$.  Assume that  $(p_j)\in\prod_{j=1}^d(0,r_j]$. Then the following assertions are equivalent:
\begin{enumerate}
    \item (Norm inequality) We have
    $$
    \norm{\prod_{j=1}^d \abs{T_jf_j}^{p_j\theta_j} }_{L^q}\leq \prod_{j=1}^d \norm{f_j}_{L^{r_j}}^{p_j\theta_j} \quad\text{for all $f_j\in L^{r_j}$.}
    $$
    \item (Existence of factorisation) For every $g\in L^{q'}$ there exist measurable functions $g_j$ such that $\abs{g}\leq \prod_{j=1}^d \abs{g_j}^{\theta_j}$ almost everywhere and such that for each $j=1,\ldots,d$ we have
  $$\left( \int \abs{T_jf_j}^{p_j} \abs{g_j}  \right)^{1/p_j}\lesssim_{p_j,r_j} \|g\|_{q'}\norm{f_j}_{L^{r_j}}\quad\text{for all $f_j\in L^{r_j}$.} 
$$
\end{enumerate}
Furthermore, the set $\prod_{j=1}^d(0,r_j]$ of admissible exponents is sharp in that outside it the equivalence may fail. 
\end{theorem}
\noindent The corresponding multilinear duality theorem for general linear operators $T_j$ is similar, but the conditions on the exponents $(p_j)$ are more complicated; see Theorem 5.2 of \cite{chv2} for the precise statement.

The bilinear case of the basic disentanglement theorem gives a uniform approach to several classical results, such  as the Maurey factorisation theorem \cite{maurey} and the Lozanovski\u{\i} factorisation theorem \cite{lozanovskii}, which can be recovered as particular cases from it. See \cite[Theorem~1.6]{chv1} for the argument for the Maurey theorem; we present here the argument for the Lozanovski\u{\i} theorem:
\begin{corollary}[Lozanovski\u{\i} factorisation theorem \cite{lozanovskii}]Let $X$ be a K\"othe function space.\footnote{A Banach space $X$ of measurable functions is called a K\"othe space
if for all $f \in \mathcal{M}$ and $g \in X$  with $|f| \leq |g|$ a.e. we have $f\in X$ and $\|f\|_X \leq \|g\|_X$, and if there exists $f \in X$ with $f > 0$ a.e. (strong saturation). The K\"othe dual $X'$ to a K\"othe space $X$ is defined as $\{g \in \mathcal{M} \, : \, \int |fg| < \infty \mbox{ for all }f \in X$\} with its natural norm. We say that $X'$ is norming if for all $f \in X$ we have $\|f\|_X = \sup_{\|g\|_{X'} \leq 1} |\int fg|$.}
Assume that its K\"othe dual $X'$ is norming. Then for each $f\in L^1$ there exist $g\in X$ and $h\in X'$ such that $f=g h$ and such that $\norm{g}_X\norm{h}_{X'}\leq \norm{f}_{L^1}$.
\end{corollary}

\begin{proof}
We establish the following equivalent statement: for each $f\in L^2$ there exist $g\in X$ and $h\in X'$ such that $f=g^{1/2} h^{1/2}$ and such that $\norm{g}_X\leq \norm{f}_{L^2}$ and $\norm{h}_{X'}\leq \norm{f}_{L^2}$. Let $f\in L^2$. Now, by H\"older's inequality,
$$
\int \tilde{g}^{1/2} \tilde{h}^{1/2} f\,\dmu\leq \int \tilde{g} \tilde{h} \, \dmu \norm{f}_{L^2}\leq  \norm{f}_{L^2} \norm{\tilde{g}}_X \|{\tilde{h}}\|_{X'}\leq \norm{f}_{L^2}
$$
for all $ \tilde{g} \in B_X$ and $ \tilde{h} \in B_{X'}$ where $B_X$ and $B_{X'}$ are the unit balls in $X$ and $X'$ respectively.

The unit balls $B_X$ and $B_{X'}$ are convex and $B_X$ is strongly saturating by the definition of a K\"othe function space. Moreover, the unit ball $B_{X'}$ is also saturating because, by assumption, $X'$ is norming.\footnote{This may be seen by contraposition. Indeed, assume that $B_{X'}$ is not saturating. Then, by definition of saturation, there exists a measurable set $E$ of positive measure such that $1_Eg=0$ almost everywhere for every $g\in B_{X'}$. Moreover, by the definition of a K\"othe function space, there exists $f\in X$ such that $f>0$ almost everywhere. Write $f':=1_Ef$. Now, $0\neq f' \in X$ and hence $\norm{f'}_X\neq 0$, but $\sup_{g\in B_{X'}} |\int  f' g  |=0$.  Therefore, $X'$ is not norming.} 

Therefore, by the basic disentanglement theorem  (with respect to the measure $f\dmu$, which is $\sigma$-finite), there exist measurable functions $\phi$ and $\psi$ such that $\phi^{1/2} \psi^{1/2}=1$ a.e. on $\{f>0\}$ and such that
$
\int \tilde{g} \phi f\dmu\leq \norm{f}_{L^2}$ for all $\tilde{g}\in B_X$ and $\int \tilde{h} \psi f\dmu\leq \norm{f}_{L^2}$ for all $\tilde{h}\in B_{X'}$. The first conclusion implies that $f=(\phi f)^{1/2} (\psi f)^{1/2}$, and the second conclusion is equivalent, by duality together with the assumption that $X'$ is norming, to the estimates $\norm{\phi f}_{X'}\leq \norm{f}_{L^2}$ and $\norm{\psi f}_{X}\leq \norm{f}_{L^2}$.

\end{proof}
The basic disentanglement theorem, Theorem \ref{theorem:basic}, has several equivalent formulations, and the equivalences are straightforward \cite[Section~2]{chv2}.
The theorem, in one of its equivalent formulations, was proved as Theorem~1.3 in \cite{chv1} using minimax theory and rather difficult compactness arguments involving finitely additive measures together with the Yosida--Hewitt theory of such measures. 

The main purpose of this paper is to provide an alternate approach which uses perturbation and different, rather easier, compactness arguments. The proof in this paper is quite short and also elementary in that it uses only the basics of measure theory and functional analysis.

For the proof in this paper, the most relevant of the equivalent formulations of Theorem~\ref{theorem:basic} are the case of  exponents $r_j =p_j = q = 1$ in Theorem \ref{theorem:duality}, and the following result which we now state:

\medskip

\begin{theorem}[Equivalent formulation of the basic disentanglement theorem, cf. Theorem 2.1 of \cite{chv2}]\label{theorem:disentanglement}Let $(\Omega,\mu)$ be a $\sigma$-finite measure space. Let $\theta_j\in(0,1)$ be such that $\sum_{j=1}^d \theta_j=1$. Let $K_j$ be indexing sets and $\{u_{j,k_j}\}_{k_j\in K_j}$ be saturating families of non-negative measurable functions indexed by them. Assume that 
\begin{equation}
\label{eq:dt_assumption}
\int_\Omega \prod_{j=1}^d \abs{\sum_{k_j\in K_j} \alpha_{j,k_j} u_{j,k_j}}^{\theta_j} \dmu \leq A \prod_{j=1}^d \left( \sum_{k_j\in K_j} \abs{\alpha_{j,k_j}}\right)^{\theta_j}
\end{equation}
for all finitely supported families $\{\alpha_{j,k_j}\}_{k_j\in K_J}$ of reals. Then there exist non-negative measurable functions $\phi_j$ such that
\begin{equation}
\label{eq:dt_decomposition}
    1\leq \prod_{j=1}^d \phi_j^{\theta_j} \quad\text{$\mu$-a.e.}
\end{equation}
and such that for each $j=1,\ldots,d$ we have
\begin{equation}\label{eq:dt_componentwise_bound}
    \int u_{j,k_j} \phi_j \dmu \leq A \quad\text{for all $k_j\in K_j$}.
\end{equation}
\end{theorem}

\subsection{The approach in this paper}
The proof in this paper proceeds in two steps. First, in Section \ref{sec:finite}, we prove a finite-dimensional version of the theorem via perturbation and  strong  compactness (which from the viewpoint of Theorem~\ref{theorem:duality} takes place on the domain side). For technical reasons we introduce an auxiliary parameter $q < 1$ for this result.
 Our use of perturbation is similar to its deployment in Pisier's proof \cite{pisier} of the Maurey factorisation theorem \cite{maurey} and in Gillespie's proof \cite{gillespie} of the Lozanovski\u{\i} factorisation theorem \cite{lozanovskii}.
Second, in Section \ref{sec:infinite}, we prove the full infinite-dimensional theorem building on its finite-dimensional version, via  the finite intersection property and weak compactness (which from the viewpoint of  Theorem~\ref{theorem:duality} takes place on the target side)\footnote{For the reader's convenience, we note that we refer to strong (i.e. ~norm-)compactness in the finite-dimensional setting,
and to weak compactness in the setting of infinite-dimensional $L^p$ spaces, in which case will always be working in the reflexive range $1 < p < \infty$.}.
Section \ref{sec:preliminaries} consists of preliminaries. In Section \ref{sec:saturation}, the definition and basic properties of {\it saturation} and {\it strong saturation} are given. In Section \ref{sec:upgrading}, we notice that in the Disentanglement Theorem \ref{theorem:disentanglement} we may make slightly stronger assumptions without loss of generality: we may assume that the measure is a probability (in place of merely $\sigma$-finite) measure and that each family  of functions is strongly saturating (in place of merely saturating).

\subsection{Compactness arguments -- a new viewpoint}\label{sec:powertrick}The use of compactness in our current proof differs from its use in our previous proof. The minimax approach of \cite{chv1} and the perturbation approach of this paper both necessitate searching for the $d$-tuple $(\phi_j)$ of functions in a compact topological space. In which space can we hope to find it? 

Conclusion \eqref{eq:dt_componentwise_bound} asserts that for the fixed weights $u_{j,k_j}=:w_j$ the functions $(\phi_j)$ satisfy
\begin{equation}
\label{eq:eq:l1_bound}
\norm{(\phi_j)}_{L^1(w_1)\times \cdots \times L^1(w_d)}:=\max_{j=1,\ldots, d} \int \abs{\phi_j} w_j \dmu \leq A.
\end{equation}
Thus, we know {\em a posteriori} that the $d$-tuple $(\phi_j)$ of functions will belong to the norm closed unit ball of the non-reflexive normed product space $L^1(w_1)\times \cdots \times L^1(w_d)$, which fails to be compact in general (in any reasonable topology). The unit ball, however, embeds in the unit ball of the bidual space $(L^\infty(w_1))^{*}\times \cdots \times (L^\infty(w_d))^{*}$ of finitely additive measures, which in turn is weak*-compact. This observation leads to the compactness approach of the first paper: we first search for finitely additive measures satisfying \eqref{eq:eq:l1_bound} -- re-interpreted for finitely additive measures (in place of functions) -- in the space of finitely additive measures equipped with the weak*-topology, and then apply to them the Yosida--Hewitt theory of finitely additive measures to eventually locate functions $(\phi_j)$ satisfying \eqref{eq:eq:l1_bound}.

In this paper, we observe that instead of searching for functions $(\phi_j)$, we can instead search for their powers $(\psi_j):=(\phi_j^{1/p})$, where $p \in (1,\infty)$ is any auxiliary exponent, in which case conclusion \eqref{eq:dt_componentwise_bound} reads
\begin{equation*}
  \int u_{j,k_j} \phi_j \dmu=\int u_{j,k_j} \psi_j^p \dmu \leq A \quad\text{for all $k_j\in K_j$}.
\end{equation*}
This asserts that for fixed weights $u_{j,k_j}=:w_j$ the functions $(\psi_j)$ satisfy
\begin{equation*}
\label{eq:lp_bound}
\norm{(\psi_j)}_{L^p(w_1)\times \cdots \times L^p(w_d)}^p:=\max_{j=1,\ldots,d} \int \abs{\psi_j}^p w_j\leq A
\end{equation*}
in place of \eqref{eq:eq:l1_bound}.
Thus, we know {\em a posteriori} that the $d$-tuple  $(\psi_j)$ of powers will belong to the norm closed ball of the reflexive normed product space $L^p(w_1)\times \cdots \times L^p(w_d)$, which is weakly compact. This observation leads to the compactness approach of the present paper: we search for powers $(\psi_j)$ of functions (in place of functions themselves) in the space $L^p(w_1)\times \cdots \times L^p(w_d)$ equipped with the weak topology. The point of the approach is to use the reflexive range of Lebesgue spaces $L^p$, in place of $L^1$; it transpires that any choice $p\in(1,\infty)$ will work.

Each of the two approaches has its advantages and disadvantages: the approach of this paper  is quite simple but  quite specific as it relies on the problem's compatibility with raising to powers; by contrast, the previous approach via finitely additive measures together with the Yosida--Hewitt theory is quite complicated but quite generic as it is applicable also in problems lacking such compatibility.

\stoptocentries
 \subsection*{Acknowledgements} A.C. was partially supported by Grant CEX2019-000904-S funded by MCIN/AEI/10.13039/501100011033 while visiting ICMAT in Madrid, and by a Leverhulme Fellowship under which part of this research was conducted. T. S. H. is supported by the Academy of Finland (through Projects 297929, 314829, 332740, and 336323). 

\starttocentries
\section{Preliminaries on saturation}\label{sec:preliminaries}

\subsection{Saturation}\label{sec:saturation}

\begin{definition}[Saturation]A non-empty collection $\cs$ of measurable sets {\it saturates} a measure space $(\Omega,\mu)$ if for every measurable set $E$ the following implication holds:
$$
\text{ if
$\mu(S\cap E)=0$ for every set $S\in\cs$, then $\mu(E)=0.$}
$$
\end{definition}
Similarly, a non-empty collection $U$ of measurable functions {\it saturates} a measure space $(\Omega,\mu)$ if the collection $\{\omega \, : \, u(\omega) \neq 0\}, u\in U$, of measurable sets saturates it, or in other words, if for every measurable set $E$ the following implication holds: 
$$
\text{if for every function $u\in U$ we have $u=0$ $\mu$-a.e. on $E$, then $\mu(E)=0$.}
$$
In addition, a collection $U$ of measurable functions {\it strongly saturates} a measure space $(\Omega,\mu)$ if there exists $u\in U$ such that $u>0$ $\mu$-almost everywhere on $\Omega$. Finally, an operator $T:X\to \cm(\Omega)$ {\it (strongly) saturates} a  measure space $(\Omega,\mu)$ if the collection $\{Tf \}_{ f\in X}$ of measurable functions (strongly) saturates it.

\begin{lemma}[Equivalence between saturation and countable cover]\label{lemma:saturation}Let $(\Omega,\mu)$ be a $\sigma$-finite measure space and $\cs$ a non-empty collection of measurable sets. Then the following assertions are equivalent:
\begin{enumerate}
    \item (Saturation) For every measurable set $E$ we have $\mu(E)=0$ whenever $\mu(E\cap S)=0$ for all $S\in\cs$.
    \item (Countable cover) There exists a countable subcollection $\{S_k\}_{k=0}^\infty$ of sets such that $\Omega=\bigcup_{k=0}^\infty S_k$ up to a set of measure zero.
\end{enumerate}
\end{lemma}
\begin{proof}We prove that saturation implies countable cover; the converse is clear by countable subadditivity of measures.
Without loss of generality, we may assume that $(\Omega,\mu)$ is finite and that $\emptyset \in \cs$. We choose the sets $S_k$ recursively by a greedy algorithm:
\begin{itemize}
    \item We choose $S_0:=\emptyset$.
    \item Given the sets $S_0,\ldots,S_k$, we define 
    $$
    m_{k+1}:=\sup_{S\in\cs} \mu(S\setminus \bigcup_{l=0}^k S_k)
    $$
    and we choose $S_{k+1}$ as a set such that $ \mu(S_{k+1}\setminus \bigcup_{l=0}^k S_k)\geq \frac{1}{2}m_{k+1}$.
\end{itemize}
It remains to check the equality$$
\mu(S\cap (\Omega\setminus \bigcup_{l=0}^\infty S_l))=0 \quad\text{ for every $S\in \cs$},
$$
which by saturation implies that $\mu (\Omega\setminus \bigcup_{l=0}^\infty S_l)=0$ and thereby completes the proof. 

Now, we check the equality. By the definition of the sets $S_k$ and because the sets $S_k\setminus \bigcup_{l=0}^{k-1} S_l$ are pairwise disjoint, we have
$$
\sum_{k=0}^\infty m_k\leq 2 \sum_{k=0}^\infty \mu(S_k\setminus \bigcup_{l=0}^{k-1} S_l)\leq 2 \mu(\Omega)<\infty
$$
and hence $\lim_{k\to \infty} m_k=0.$ Therefore, for every $S\in\cs$, we have
\begin{equation*}
    \begin{split}
   &\mu(S\cap (\Omega\setminus \bigcup_{l=0}^\infty S_l))=\mu(S\setminus \bigcup_{l=0}^\infty S_l)
   \leq \lim_{k\to \infty} \mu(S\setminus \bigcup_{l=0}^k S_l)\leq \lim_{k\to \infty} \sup_{S\in\cs} \mu(S\setminus \bigcup_{l=0}^k S_l)\\
   &= \lim_{k\to \infty}  m_{k+1}=0.
    \end{split}
\end{equation*}
\end{proof}
\begin{corollary}[Preservation of saturation]\label{corollary:saturation}Let $(\Omega,\mu)$ be a $\sigma$-finite measure space. Let $\{u_j:\Omega\to \br\}_{u_j\in U_j}$, $j=1,\ldots,d$, be saturating collections of  measurable functions.  Let $v:\br^d \to \br$ be a Borel-measurable function with the property that $v(a_1,\ldots,a_d)\neq 0$ whenever $a_1\neq 0,\ldots, a_d\neq 0$. Then the collection $\{v(u_1,\ldots,u_d):\Omega \to \br \}_{u_1\in U_1,\ldots,u_d\in U_d}$ of functions is saturating.
\end{corollary}
\begin{proof}Since each collection $\{u_j:\Omega\to \br\}_{u_j\in U_j}$ of functions is saturating, by the implication  (1) $\implies$ (2) of Lemma \ref{lemma:saturation}, there exists $\{u_{j,k_j}\}_{k_j=1}^\infty$ such that $\bigcup_{k_j=1}^\infty \{\omega : u_{j,k_j}(\omega)\neq 0\}=\Omega$  up to a set of measure zero. Therefore, by  the assumption on the function $v:\br^d\to\br$, we have,  up to a set of measure zero,
\begin{equation*}
    \begin{split}
        &\bigcup_{k_1=1}^\infty \cdots \bigcup_{k_d=1}^\infty \{\omega: v(u_{1,k_1}(\omega),\ldots,u_{d,k_d}(\omega))\neq 0\}\\
        &\supseteq  \bigcup_{k_1=1}^\infty \cdots \bigcup_{k_d=1}^\infty \{\omega : u_{1,k_1}(\omega)\neq 0\}\cap \cdots\cap \{\omega: u_{d,k_d}(\omega)\neq 0\}\\
        &=\left(\bigcup_{k_1=1}^\infty \{\omega: u_{1,k_1}(\omega)\neq 0\} \right)\cap \cdots \cap \left(\bigcup_{k_d=1}^\infty \{\omega: u_{d,k_d}(\omega)\neq 0\} \right)=\Omega.
    \end{split}
\end{equation*}
Therefore, by the implication   (2) $\implies$ (1) of Lemma \ref{lemma:saturation}, the collection $\{v(u_1,\ldots,u_d):\Omega \to \br \}_{u_1\in U_1,\ldots,u_d\in U_d}$ of functions is saturating.
\end{proof}
\subsection{Upgrading assumptions: probability measure and strong saturation}\label{sec:upgrading}

\begin{lemma}\label{lemma:upgrading}In the Disentanglement Theorem \ref{theorem:disentanglement} we may assume without loss of generality that the measure is a probability (in place of merely $\sigma$-finite) measure and that each family of functions is strongly saturating (in place of merely saturating).
\end{lemma}
\begin{proof}We prove that, given a  $\sigma$-finite measure $\mu$ and  saturating families $U_j$ of functions, we can define an auxiliary {\it probability } measure $\nu$ and auxiliary {\it strongly saturating} families $V_j$ while preserving each of the inequalities \eqref{eq:dt_assumption}, \eqref{eq:dt_decomposition}, and \eqref{eq:dt_componentwise_bound} in the sense that each of them holds for $\mu$ and $U_j$ if and only if it holds for $\nu$ and $V_j$.  Then, the lemma is immediate from the following diagram of implications:

\medskip 

{
\centering
\begin{tikzpicture}

  \node[text width=30mm] (eqau) {Eq. \eqref{eq:dt_assumption} for $\{U_j\}$ and $\mu$};
  \node[right=of eqau,text width=35mm] (eqav) {Eq. \eqref{eq:dt_assumption} for $\{V_j(U,\mu)\}$ and $\nu(U,\mu)$};
   \node[below=of eqav,text width=35mm] (eqcv) {Eqs. \eqref{eq:dt_decomposition} and \eqref{eq:dt_componentwise_bound} for $\{V_j(U,\mu)\}$ and $\nu(U,\mu)$};
\node[below=of eqau,text width=30mm] (eqcu) {Eqs. \eqref{eq:dt_decomposition} and \eqref{eq:dt_componentwise_bound} for $\{U_j\}$ and $\mu$};
 
  \node[between=eqav and eqcv] (dummy1) {};
  \node[right=of dummy1,text width=55mm] (dummy2) {Thm \ref{theorem:disentanglement} for probability measures and strongly saturating functions};

  \draw[implies-implies,double equal sign distance] (eqau) to (eqav);
   \draw[implies-implies,double equal sign distance] (eqcu) to (eqcv);
   \draw[-implies,double equal sign distance,bend left] (eqav) to (eqcv);
\end{tikzpicture}
}

We conclude the proof by defining the auxiliary families and the measure.
Let $U_j$ be saturating families of functions and $\mu$ a $\sigma$-finite measure.
Because of $\sigma$-finiteness, we can find $w>0$ $\mu$-a.e. such that $\int w \dmu=1$. Because each family $\{u_{j,k_j}\}_{k_j\in K_j}$ is saturating, by Lemma \ref{lemma:saturation}, we can find a countable subfamily $\{u_{j,n}\}_{n=1}^\infty$ such that $u_{j}:=\sum_{n=1}^\infty  2^{-n}u_{j,n}>0$ $\mu$-almost-everywhere. 

Now, we define the measure $\dnu:=w\dmu$ and families $V_j:=\{w^{-1}u_{j,k_j}\}_{k_j\in K_j}\cup \{w^{-1}u_j\}$. By definition, the measure $\nu$ is a probability measure and each family $V_j$ is strongly saturating as it includes the almost everywhere strictly positive function $w^{-1}u_j$. Furthermore, by writing out the definitions of $V_j$ and $\nu$, we observe that the inequality \eqref{eq:dt_assumption} holds for $U_j$ and $\mu$ if and only if it holds for $V_j$ and $\nu$.  Indeed, the inequality for $U_j$ and $\mu$ implies the inequality for $V_j$ and $\nu$ as follows:
\begin{equation*}
    \begin{split}
        &\int_\Omega \prod_{j=1}^d \abs{\sum_{k_j\in K_j} \alpha_{j,k_j} (w^{-1}u_{j,k_j})+\beta_{j} (w^{-1} u_j)}^{\theta_j} w\dmu\\
        &=\int_\Omega \prod_{j=1}^d \abs{\sum_{k_j\in K_j} \alpha_{j,k_j} u_{j,k_j}+\beta_{j}\sum_{n=1}^\infty 2^{-n}u_{j,n}}^{\theta_j} \dmu\\
        &        \leq A \prod_{j=1}^d \left( \sum_{k_j\in K_j} \abs{\alpha_{j,k_j}}+ \abs{\beta_j}\sum_{n=1}^\infty 2^{-n}\right)^{\theta_j}\\
        &= A\prod_{j=1}^d \left( \sum_{k_j\in K_j} \abs{\alpha_{j,k_j}}+\abs{\beta_j}\right)^{\theta_j}.
    \end{split}
\end{equation*}
(The converse implication is clear.)
Similarly, given functions $\phi_j$, each of the inequalities  \eqref{eq:dt_decomposition} and \eqref{eq:dt_componentwise_bound} holds for $U_j$ and $\mu$ if and only if it holds for $V_j$ and $\nu$.

\end{proof}

\section{Finitistic case via perturbation and  compactness}\label{sec:finite}
 In this section we introduce and prove the following finite-dimensional variant of Theorem \ref{theorem:disentanglement}:
\begin{proposition}\label{prop:functional_finite}Let $(\Omega,\mu)$ be a $\sigma$-finite measure space. Let $q\in(0,1)$ and  $\theta_j\in(0,1)$ be such that $\sum_{j=1}^d \theta_j=1$. Let $K_j$ be finite indexing sets and $\{u_{j,k_j}\}_{k_j\in K_j}$ be saturating families of non-negative measurable functions indexed by them. Assume that 
\begin{equation}\label{eq:intermediate_maurey}
\int_\Omega \prod_{j=1}^d \abs{\sum_{k_j\in K_j} \alpha_{j,k_j} u_{j,k_j}}^{\theta_jq} \dmu \leq A \prod_{j=1}^d \left( \sum_{k_j\in K_j} \abs{\alpha_{j,k_j}}\right)^{\theta_jq}
\end{equation}
for all families $\{\alpha_{j,k_j}\}_{k_j\in K_j}$ of reals. Then there exist non-negative measurable functions $\phi_j$ such that
\begin{equation}
   \int \left( \prod_{j=1}^d \phi_j^{\theta_j}\right)^{q'}\dmu\leq A
    \label{eq:maurey_functions_conclusion_a}
\end{equation}
and such that for each $j=1,\ldots,d$ we have
\begin{equation}\label{eq:componentwise_bound}
    \int u_{j,k_j} \phi_j \dmu \leq A \quad\text{for all $k_j\in K_j$}.
\end{equation}
\end{proposition}
The conjugate exponent $q'$ of $q\in(0,1]$ is defined  by  $q':=\frac{q}{q-1}$ when $q\in(0,1)$ and $q':=-\infty$ (not $+ \infty$) when $q=1$.
 When $q=1$, inequality \eqref{eq:maurey_functions_conclusion_a} is interpreted as  $\essinf \prod_{j=1}^d \phi_j^{\theta_j}\geq 1.$ Thus, Theorem \ref{theorem:disentanglement} corresponds to the case $q=1$ and arbitrary indexing sets, whereas Proposition ~\ref{prop:functional_finite} corresponds to the case $0<q<1$ and finite indexing sets. Theorem~\ref{theorem:disentanglement} implies Proposition~\ref{prop:functional_finite} by an easy argument\footnote{For the passage from $q=1$ to $0<q<1$ we argue as in the proof of \cite[Theorem 2.3]{chv1} by introducing a ``dummy'' parameter as follows. We observe that the $d$-dimensional estimate of the form \eqref{eq:intermediate_maurey} can be viewed as the $(d+1)$-dimensional estimate of the form \eqref{eq:dt_assumption} by writing the former estimate as
$$
\int_\Omega \prod_{j=1}^{d} \abs{\sum_{k_j\in K_j} \alpha_{j,k_j} u_{j,k_j}}^{\theta_jq} \abs{\sum_k \beta_k 1}^{1-q}\dmu \leq A \prod_{j=1}^{d+1} \left( \sum_{k_j\in K_j} \abs{\alpha_{j,k_j}}\right)^{\theta_jq} \abs{\sum_k \beta_k}^{1-q}. 
$$
Feeding the corresponding choice $\tilde{\theta}_j:=\theta_j q$ and $\{\tilde{u}_j\}:=\{u_j\}$ for $j=1,\ldots,d$ and $\tilde{\theta}_{d+1}:=(1-q)$ and $\{\tilde{u}_{d+1}\}:=\{1\}$ into Theorem \ref{theorem:disentanglement} gives the existence of $(\tilde{\phi}_j)$ such that $$\prod_{j=1}^{d+1}\tilde{\phi}_j^{\tilde{\theta_j}}\geq 1  \quad\text{ and } \quad 
\int \tilde{u}_{j,k_j}\tilde{\phi_j} \dmu\leq A,$$ 
inequalities which are written out as
\eqref{eq:maurey_functions_conclusion_a} and \eqref{eq:componentwise_bound}. The passage from arbitrary to finite indexing set is of course immediate.}.
The converse implication (consisting of the passage both from $q<1$ to $q=1$ and from finite indexing sets to arbitrary indexing sets) is more difficult and is proved via the finite intersection property and compactness 
in Section \ref{sec:infinite}. 

Thus, the original problem of proving Theorem \ref{theorem:disentanglement} can be reduced to the problem of proving Proposition \ref{prop:functional_finite}.
The point of introducing this auxiliary problem is that it is more amenable to a variational argument than the original problem is: the condition that the indexing set is finite (in place of arbitrary) ensures the existence of a maximiser, while the condition $0<q<1$ (in place of $q=1$) ensures that every maximiser is non-vanishing almost everywhere. 

The solution to the auxiliary problem is most naturally written in the language of positive operators on lattices\footnote{The disentanglement theorem can be phrased in terms of positive operators on lattices  or, equivalently, in terms of families of functions. (The explicit formulations and their  equivalences are given in \cite[Section 2]{chv2}, but are not needed in this paper.)
Each formulation affords a natural case of a finite character: the case of finite families of functions in function-formulation, the case of finite-dimensional lattices in operator-formulation. 

In this paper we work with both the operator and function formulations. However, in Section \ref{sec:infinite}, the fundamental structure of the problem becomes apparent, and in this context it is much more natural to work in the function formulation. This structure reveals itself in the following essential difference in finite closure: whereas the smallest family containing  finitely many finite families of functions remains finite, the smallest sub-vector-lattice containing finitely many finite-dimensional sub-vector-lattices may fail to remain finite-dimensional. It is because of this closure property for families of functions we can use the finite intersection property in Section \ref{sec:infinite}.

Nevertheless in this Section \ref{sec:finite}, the proof is more naturally written in terms of operators (as in Proposition \ref{proposition:maurey}) rather than in terms of functions (as in Proposition \ref{prop:functional_finite}). The particular case with $T_j:\ell^1(K_j) \to \cm(\Omega)$ given by $T_j(\{\alpha_{j,k_j}\}_{k_j\in K_j}):=\sum_{k_j\in K_j} \alpha_{j,k_j} u_{j,k_j}$ of Proposition \ref{proposition:maurey} then corresponds precisely to Proposition \ref{prop:functional_finite}.
}:

\medskip

\begin{proposition}[cf. Theorem 2.3. in \cite{chv1}]\label{proposition:maurey}Let $(\Omega,\mu)$ be a $\sigma$-finite measure space. Let $q\in(0,1)$ and $\theta_j\in(0,1)$ be such that $\sum_{j=1}^d \theta_j=1$. Let each $X_j$ be a  finite-dimensional normed lattice and each $T_j:X_j\to \cm(\Omega)$ be a saturating positive linear operator.  Assume that
\begin{equation}\label{eq:finite_a}
    \int \prod_{j=1}^d \abs{T_jf_j}^{\theta_jq}\dmu\leq A \prod_{j=1}^d \norm{f_j}_{X_j}^{\theta_jq}\quad\text{for all $f_j\in X_j$}.
\end{equation}
Then there exist non-negative measurable functions $\phi_j$ such that
\begin{equation}
    \int \left( \prod_{j=1}^d \phi_j^{\theta_j}\right)^{q'}\dmu\leq A
    \label{eq:maurey_conclusion_a}
\end{equation}and such that for each $j=1,\ldots,d$ we have
\begin{equation}
      \int \abs{T_jf_j} \phi_j \dmu \leq A \norm{f_j}_{X_j} \quad \text{ for all $f_j\in X_j$}. \label{eq:maurey_conclusion_b}
\end{equation}
 
\end{proposition}

\begin{remark} Proposition \ref{proposition:maurey} is called a {\em Multilinear Maurey factorisation theorem for positive operators on finite-dimensional lattices} in the the nomenclature of our previous papers \cite{chv1,chv2}.
\end{remark}
\begin{proof}[Proof of Proposition \ref{proposition:maurey}] 
In the case that each space $X_j$ is one-dimensional, we can take any non-zero vectors $g_j\in X_j$ with $g_j\geq 0$ and $\norm{g_j}_{X_j}\leq 1$ and write assumption \eqref{eq:finite_a} as
$$
\int \prod_{j=1}^d \abs{T_jg_j}^{\theta_jq}\dmu=\int \abs{T_jg_i}\cdot \left( \frac{1}{\abs{T_jg_i}}\prod_{j=1}^d \abs{T_jg_j}^{\theta_jq}\right)\dmu \leq A;
$$
we observe that conclusions \eqref{eq:maurey_conclusion_a} and \eqref{eq:maurey_conclusion_b} are trivially satisfied by the functions $(\phi_j)$ given by
\begin{equation}
    \label{eq:formula_phi}
\phi_i(g):=\frac{1}{T_ig_i}\left(\prod_{j=1}^d (T_jg_j)^{q\theta_j}\right).
\end{equation}

We prove that, in the finite-dimensional case, the conclusion of Proposition~\ref{proposition:maurey} is satisfied by the functions $\phi_j(g)$ defined by this same formula \eqref{eq:formula_phi} for {\em any} maximiser $(g_j)$. The proof proceeds by a sequence of four claims: first, in Claim \ref{claim_1}, we prove that a maximiser $(g_j)$ for inequality \eqref{eq:finite_a} exists. Second, in Claim \ref{claim_2}, we check that $(T_jg_j)>0$ and so we can define $\phi_i(g)$ by formula \eqref{eq:formula_phi} since no division by zero occurs. Third in Claim \ref{claim_3}, we observe that the inequality  \eqref{eq:maurey_conclusion_a} is satisfied by $(\phi_i(g))$ due its specific defining formula. Finally, in Claim \ref{claim_4}, we prove via perturbation that for each $j$, inequality \eqref{eq:maurey_conclusion_b} for general $f_j \in X_j$ is satisfied by $\phi_j(g)$.

\begin{claim}[Maximiser exists]\label{claim_1} Assume that $A$ is the least constant in the hypothesis \eqref{eq:finite_a}. Then there exist $g_j\in X_j$ with $0\leq g_j$ and $\norm{g_j}_{X_j}\leq 1$ such that 
$$\int \prod_{j=1}^d \abs{T_j g_j}^{\theta_jq}\dmu=A.$$
\end{claim}
\begin{proof}[Proof of Claim \ref{claim_1}] We define the function $I:X_1\times \cdots \times X_d\to \br$ by setting $$I(f_1,\ldots,f_d):=\int \prod_{j=1}^d \abs{T_jf_j}^{\theta_jq}\dmu.$$ 
As usual, we equip the product space $X_1\times \cdots \times X_d$ with the norm $\norm{(f_1,\ldots,f_d)}_{X_1\times\cdots\times X_d}:=\max_{j=1,\ldots,d} \norm{f_j}_{X_j}$. The function $I:X_1\times \cdots \times X_d\to \br$ is continuous on the normed product space $X_1\times \cdots X_d$ because of the estimate
\begin{equation*}
    \begin{split}
        &\abs{I(f_1,\ldots,f_d)-I(g_1,\ldots,g_d)}\\
        &\leq \sum_{k=1}^d \abs{I(g_1,\ldots,g_{k-1},f_k,f_{k+1},\ldots,f_d)-I(g_1,\ldots,g_{k-1},g_k,f_{k+1},\ldots,f_d)}\\
        &\leq \sum_{k=1}^d I(g_1,\ldots,g_{k-1},(f_k-g_k),f_{k+1},\ldots,f_d)\\
        &\leq A \sum_{k=1}^d \prod_{j=1}^{k-1} \norm{g_j}_{X_j}^{\theta_jq} \cdot \norm{f_k-g_k}_{X_k}^{\theta_kq} \cdot \prod_{j=k+1}^d \norm{f_j}_{X_j}^{\theta_jq}.
    \end{split}
\end{equation*}
Using standard notation, we write $(B_{X_j})_+:=\{ f_j\in X_j : f_j\geq 0 \text{ and } \norm{f_j}_{X_j}\leq 1\}$. The product set $(B_{X_1})_+\times \cdots \times (B_{X_d})_+$ is compact because it is a product of compact sets.

By continuity and compactness, there exists $(g_1,\ldots,g_d)\in (B_{X_1})_+\times \cdots \times (B_{X_d})_+$ such that
$$
\int \prod_{j=1}^d \abs{T_jg_j}^{\theta_jq}\dmu=\sup_{f_j\in (B_{X_j})_+}\int \prod_{j=1}^d \abs{T_jf_j}^{\theta_jq}\dmu.
$$
Because of positivity and homogeneity, and because $A$ is the least constant in the hypothesis, we have
$$
\sup_{f_j\in (B_{X_j})_+}\int \prod_{j=1}^d \abs{T_jf_j}^{\theta_jq}\dmu=\sup_{f_j\in X_j} \frac{\int \prod_{j=1}^d \abs{T_jf_j}^{\theta_jq}\dmu}{ \prod_{j=1}^d \norm{f_j}_{X_j}^{\theta_jq}}=A.
$$
The proof of the claim is completed.
\end{proof}

Now, we define $\phi_i(g)$ by formula \eqref{eq:formula_phi} for a maximiser $(g_j)$, where no division by zero occurs thanks to Claim \ref{claim_2}.

\begin{claim}[Maximising function is strictly positive almost everywhere]\label{claim_2}Assume that $(g_1,\ldots,g_d)\in (B_{X_1})_+\times\cdots \times (B_{X_d})_+$ is a maximiser. Then for each $j=1,\ldots,d$ we have $Tg_j>0$ $\mu$-almost everywhere.
\end{claim}

\begin{proof}[Proof of Claim \ref{claim_2}]Fix vectors $(f_1,\ldots,f_d)\in (B_{X_1})_+\times\cdots\times (B_{X_d})_+$.
 By applying assumption \eqref{eq:finite_a} to the vectors $\tilde{g}_j:=g_j+\epsilon f_j$, by splitting the measure space, by positivity, and by the attainment of the constant, we obtain
\begin{equation*}
    \begin{split}
   & A (1+\epsilon)^q \geq \int \prod_{j=1}^d (T_j(g_j+\epsilon f_j))^{\theta_j q} \dmu\\
   &=\int_{\{\prod_{j=1}^d (T_jg_j)^{\theta_j q}>0\}} \prod_{j=1}^d (T_j(g_j+\epsilon f_j))^{\theta_j q} \dmu+\int_{\{\prod_{j=1}^d (T_jg_j)^{\theta_j q}=0\}} \prod_{j=1}^d (T_j(g_j+\epsilon f_j))^{\theta_j q} \dmu\\
   &\geq \int_{\{\prod_{j=1}^d (T_jg_j)^{\theta_j q}>0\}} \prod_{j=1}^d (T_jg_j)^{\theta_j q}\dmu+ \epsilon^q \int_{\{\prod_{j=1}^d (T_jg_j)^{\theta_j q}=0\}} \prod_{j=1}^d (T_jf_j)^{\theta_j q}\dmu\\
   &=A+\epsilon^q \int_{\{\prod_{j=1}^d (T_jg_j)^{\theta_j q}=0\}} \prod_{j=1}^d (T_jf_j)^{\theta_j q}\dmu.
    \end{split}
\end{equation*}
Letting $\epsilon\to 0$, and crucially using the assumption that $q<1$, yields
\begin{equation}
    \label{eq:temp5}
 \int_{\{\prod_{j=1}^d (T_jg_j)^{\theta_j q}=0\}} \prod_{j=1}^d (T_jf_j)^{\theta_j q}\dmu=0.
\end{equation}

By assumption, each family $\{T_jf_j\}_{f_j\in (B_{X_j})_+}$ of functions saturates $\Omega$ and hence, by Corollary \ref{corollary:saturation}, so does the family $H:=\{\prod_{j=1}^d (T_jf_j)^{\theta_j q} \}_{f_1\in (B_{X_1})_+,\ldots,f_d\in (B_{X_d})_+}$ of functions. By \eqref{eq:temp5}, for every $h\in H$ we have $h=0$ $\mu$-almost everywhere on the set $\{\prod_{j=1}^d (T_jg_j)^{\theta_j q}=0\}$, which implies, by definition of saturation, that $\mu(\{\prod_{j=1}^d (T_jg_j)^{\theta_j q}=0\})=0$. Therefore, we have $T_jg_j>0$ $\mu$-almost everywhere for every $j=1,\ldots,d$.  The proof of the claim is completed.
\end{proof}

In the endpoint case $q=1$, the proof of Claim~\ref{claim_1} shows that there exists a maximiser $(T_jg_j)$; however, one can construct an example of operators $T_j$ for which each component $T_jg_j$ of every maximiser vanishes on a set of positive measure. 
\begin{example}Let $T: l^1(\{0,1\}) \to \mathcal{M}(\Omega)$ be given by 
$T(\alpha, \beta) = \alpha u + \beta v$ where $u$ and $v$ are nonnegative and disjointly supported, and $\int u > \int v >0$. Let $T_j =T$ 
and $\theta_j = 1/d$ for all $j$. Then extremals for the multilinear problem
$$ \int \prod_{j=1}^d \abs{T(\alpha_j,\beta_j)}^{q/d}\dmu\leq A_q(u,v) \prod_{j=1}^d (\abs{\alpha_j}+\abs{\beta_j})^{q/d}$$
consist precisely of $d$-tuples of extremals for the linear problem $\int |T(\alpha, \beta)|^q {\rm d}\mu \leq A_q(u,v)(|\alpha| + |\beta|)^q$. But, when $q=1$, the extremum for this linear problem satisfies $A_1(u,v)=\max \{\norm{u}_{L^1},\norm{v}_{L^1}\}=\norm{u}_{L^1}$  and hence its extremals $(\alpha, \beta)$ satisfy $\beta = 0$, and so $T(\alpha, \beta)$ vanishes outside ${\rm supp }\; u$. 
\end{example}
Therefore, $(\phi_i(g))$ cannot be defined by formula \eqref{eq:formula_phi} at the endpoint $q=1$, and hence one needs to work with the exponents $q\in(0,1)$  in order to avoid possibly dividing by zero. 
\begin{example_continued}The formula \eqref{eq:formula_phi} for $\phi(q):=\phi_{i}(q)$ does not make sense directly for $q=1$ because of division by zero. 
However, the limit $\phi:=\lim_{q \uparrow 1}\phi(q)$ does make sense everywhere and satisfies $\prod_i \phi^{1/d}\geq \frac{1}{A_1(u,v)}$ and $\int T(\alpha,\beta) \phi \dmu\leq (\abs{\alpha}+\abs{\beta})$ as desired. Indeed, the limit has the explicit formula
\begin{equation*}
    \begin{split}
    \phi:=\lim_{q \uparrow 1}\phi(q)&=\lim_{q\uparrow 1}(\norm{u}_{L^q}^{-q} u^{q-1}1_{\supp u}+\norm{v}_{L^q}^{-q} v^{q-1}1_{\supp v})\\
&=\frac{1}{\norm{u}_{L^1}}1_{\supp u}+\frac{1}{\norm{v}_{L^1}}1_{\supp v}\geq \frac{1}{A_1(u,v)}1_{\Omega}.
    \end{split}
\end{equation*}
This discussion serves to illustrate the passage from the case $0<q<1$ to the case $q=1$ which takes place in Section \ref{sec:infinite}.

\end{example_continued}

Next, we check that the inequality \eqref{eq:maurey_conclusion_a} is satisfied:

\begin{claim}[Conclusion \eqref{eq:maurey_conclusion_a} is satisfied]\label{claim_3}Assume that each $g_j\in (B_{X_j})_+$ is strictly positive $\mu$-almost everywhere. We define the functions $\phi_i(g)$ by formula \eqref{eq:formula_phi}, in which no division by zero occurs by the assumption. Then the functions $\phi_i(g)$ satisfy the inequality \eqref{eq:maurey_conclusion_a}, that is $$
    \int \left( \prod_{j=1}^d \phi_j^{\theta_j}\right)^{q'}\dmu\leq A.
$$
\end{claim}
\begin{proof}[Proof of Claim \ref{claim_3}] We have the identity
$$
\prod_{i=1}^d \phi_i^{\theta_i}=\prod_{i=1}^d \left(\prod_{j=1}^d (T_jg_j)^{\theta_jq} (T_ig_i)^{-1}\right)^{\theta_i}=\left(\prod_{j=1}^d (T_jg_j)^{\theta_j} \right)^{q-1}.
$$
This together with $q':=q/(q-1)$ further gives  the identity $\left(\prod_{i=1}^d \phi_i^{\theta_i}\right)^{q'}=\prod_{j=1}^d (T_jg_j)^{\theta_jq}$. Therefore, by assumption \eqref{eq:finite_a}, we have
$$
\int \left(\prod_{i=1}^d \phi_i^{\theta_i}\right)^{q'}\dmu=\int \prod_{j=1}^d (T_jg_j)^{\theta_jq} \dmu \leq A \prod_{j=1}^d\norm{g_j}_{X_j}^{\theta_jq} \leq A.
$$
\end{proof}

Finally,  we prove via a standard variational argument that inequality \eqref{eq:maurey_conclusion_b} is satisfied:

\begin{claim}[Conclusion \eqref{eq:maurey_conclusion_b} is satisfied]\label{claim_4}Assume that $(g_1,\ldots,g_d)\in (B_{X_1})_+\times\cdots \times (B_{X_d})_+$ is a maximiser.
We define the measurable functions $\phi_j:=\phi_j(g)$ by formula \eqref{eq:formula_phi}. 
Then the functions $\phi_j$ satisfy the inequality \eqref{eq:maurey_conclusion_b}, that is
$$
  \int \abs{T_jf_j} \phi_j \dmu \leq A \norm{f_j}_{X_j} \quad \text{ for all $f_j\in X_j$}.$$
\end{claim}
\begin{proof}[Proof of Claim \ref{claim_4}]Fix a component index $i$. Fix a vector $f_i\in X_i$. By positivity and homogeneity, we may assume that $f_i\geq 0$ and that $\norm{f_j}_{X_i}\leq 1$. Let $\epsilon>0$ be a perturbation parameter. Applying assumption \eqref{eq:finite_a} to a perturbation $(\tilde{g}_j)$ in the $i$-th component around the maximiser, $$\tilde{g}_j:=\begin{cases} g_j & \text{ for } j\neq i\\ g_i+\epsilon f_i & \text{ for } j=i, \end{cases}$$
we obtain
\begin{equation}
    \label{eq:temp1}
\int \prod_{j:j\neq i} (T_j g_j)^{\theta_jq} (T_i(g_i+\epsilon f_i))^{\theta_i q}\dmu \leq A (1+\epsilon)^{\theta_i q}.
\end{equation}
By linearity,
\begin{equation}  \label{eq:temp2}
\begin{split}
    &\int \prod_{j:j\neq i} (T_j g_j)^{\theta_jq} (T_i(g_i+\epsilon f_i))^{\theta_i q}\dmu\\
    &= \int \left(\prod_{j=1}^d (T_jg_j)^{\theta_jq}\right) (\epsilon (T_ig_i)^{-1} T_if_i+1)^{\theta_iq}\dmu.  
    \end{split}
    \end{equation}
Since $(g_j)$ is a maximiser, we have
\begin{equation}
      \label{eq:temp3}
      \int \prod_{j=1}^d (T_jg_j)^{\theta_jq} \dmu=A.
\end{equation}
Combining estimates \eqref{eq:temp1} and \eqref{eq:temp2}, subtracting equality \eqref{eq:temp3} from both sides, and dividing both sides by $\epsilon$ yields
\begin{equation*}
    \begin{split}
&\int  \left(\prod_{j=1}^d (T_jg_j)^{\theta_jq}\right) \left( \frac{(1+ \epsilon (T_ig_i)^{-1} T_if_i)^{\theta_iq}-1}{\epsilon} \right)\dmu\leq A\left( \frac{(1+\epsilon)^{\theta_i q} -1}{\epsilon}\right).
\end{split}
\end{equation*}
 By Fatou's lemma, letting $\epsilon\to 0$ gives
 $$
\int \phi_i T_if_i \dmu:=  \int \left(\prod_{j=1}^d (T_jg_j)^{\theta_jq}\right) (T_ig_i)^{-1} T_if_i\dmu \leq A.
 $$
 The proof of the claim is completed.
\end{proof}
This completes also the proof of Proposition \ref{proposition:maurey}.
\end{proof}

\section{General case via finite intersection property and weak compactness}\label{sec:infinite}
In this section we establish the general case of Theorem~\ref{theorem:disentanglement}, building on the finitistic case given by Proposition \ref{prop:functional_finite}. The idea of the proof is to find an exact factorisation among approximate factorisations using the finite intersection property and weak compactness.

We recall that, by Lemma \ref{lemma:upgrading},
we may assume without loss of generality that the measure $\mu$ is a probability (in place of merely $\sigma$-finite) measure and that each family $U_j$ of functions is strongly saturating (in place of merely saturating). 

By homogeneity, we may assume without loss of generality that $A=1$. For $q\in(0,1]$ and $L_j\subseteq K_j$, we define the set $\Phi_{q,\{L_j\}}$ of approximate factorisations by
\begin{subequations}
\begin{align}
   \nonumber \Phi_{q,\{L_j\}}:=\{&(\phi_j)\in \cm(\Omega)_+^d : \eqref{eq:defa} \mbox{ and }\eqref{eq:defb} \mbox{ hold}\}
\end{align}
where
\begin{align}
\label{eq:defa}\int \left( \prod_{j=1}^d \phi_j^{\theta_j}\right)^{q'} \dmu \leq 1 \\
\label{eq:defb} \mbox{and } \max_{j=1,\ldots,d} \sup_{k_j \in L_j} \int u_{j,k_j}\phi_j \dmu \leq 1.
\end{align}
\end{subequations}    

In the case $q=1$ (or equivalently, $q'=-\infty$) inequality \eqref{eq:defa} is interpreted as $\essinf \prod_{j=1}^d \phi_j^{\theta_j} \geq 1$. 
When $q=1$ and $L_j=K_j$, the set $\Phi_{1,\{K_j\}}$ consists of exact factorisations because in this case the inequalities \eqref{eq:defa} and \eqref{eq:defb}  recover the inequalities \eqref{eq:dt_decomposition} and \eqref{eq:dt_componentwise_bound}.

By the strong saturation hypothesis, for each $j=1,\ldots,d$, there exists $u_{j,k_{j,0}}$ with $k_{j,0}\in K_j$ such that $u_{j,k_{j,0}}>0$ $\mu$-almost everywhere. We fix, once and for all, an arbitrary such weight $w_j:=u_{j,k_{j,0}}$ with $k_{j,0}\in K_j$   and define the collection
\begin{equation}
\label{eq:deftemp5}
\cl_j:=\{L_j\subseteq K_j: \text{$L_j$ is finite and $L_j\ni {k_{j,0}}$ }\}.
\end{equation}
of indexing sets.
Note that by definition  the collection $\{u_{j,k_j}\}_{k_j\in L_j}$ of functions is strongly saturating for every indexing set $L_j\in \cl_j$.

The sets of approximate factorisations have the following key properties:
\begin{claim}[Properties of the sets of approximate factorisations]\label{claim_5}
The sets $\Phi_{q,\{L_j\}}$ satisfy:
\begin{enumerate}

  \item \label{item:monotonicity}(Monotonicity) We have
    $\Phi_{q_1,\{L_{j,1}\}}\supseteq  \Phi_{q_2,\{L_{j,2}\}}$ whenever $q_1\leq q_2$ and $L_{j,1}\subseteq L_{j,2} \subseteq K_j$ for each $j$.
    
    \item (Set limit) For every set $Q\subseteq (0,1]$ of exponents and collection $\mathcal{K}_j\subseteq P(K_j)$ of indexing sets, we have
    $$
    \bigcap_{q\in Q, L_j\in \mathcal{K}_j} \Phi_{q,\{L_j\}}=\Phi_{\sup_{q\in Q} q,\{\bigcup_{L_j\in \mathcal{K}_j}L_j\}}.
    $$
    
  \item (Non-emptiness) For all $q\in(0,1)$ and $L_j\in \cl_j$, the set $\Phi_{q,\{L_{j}\}}$ is non-empty.
 \item (Finite intersection property) The collection $\{\Phi_{q,\{L_{j}\}}\}_{q\in(0,1),L_j\in \cl_j}$ has the finite intersection property.
 \end{enumerate}
 \end{claim}
 \begin{proof}[Proof of Claim \ref{claim_5}] \; \; 
 
\begin{itemize}[leftmargin=*]
\item
{\it Monotonicity.} Monotonicity with respect to exponents $q$ is immediate from the defining inequality \eqref{eq:defa} together with Jensen's inequality; monotonicity with respect to indexing sets $L_j$ is immediate from the defining inequality \eqref{eq:defb}.
 
 \item
 {\it Set limit.} We can write $
 \bigcap_{q\in Q, L_j\in \mathcal{K}_j} \Phi_{q,\{L_j\}}=\bigcap_{q\in Q} \bigcap_{L_j\in \mathcal{K}_j} \Phi_{q,\{L_j\}}.$
 The set limit  with respect to indexing sets $L_j$ is immediate from the defining inequality \eqref{eq:defb}. 
    The set limit with respect to exponents $q$ follows from the defining inequality \eqref{eq:defa} as follows: we pick a sequence $q_n\in Q$ with $q_n\uparrow \sup_{q\in Q} q$ and then, in the case $\sup_{q\in Q} q\in(0,1)$, use Fatou's lemma or, in the case $\sup_{q\in Q} q=1$, use the fact that
    $$
    \int \left( \prod_{j=1}^d \phi_j^{\theta_j}\right)^{q_n'} \dmu \leq 1 \text{ for all $q_n\in Q$ with $q_n\uparrow 1$} \iff \essinf \prod_{j=1}^d \phi_j^{\theta_j} \geq 1.
    $$
    
 \item
 {\it Non-emptiness.} By Jensen's inequality together with the assumption that $\mu$ is a probability measure, and by assumption \eqref{eq:dt_assumption}, we have 
    $$
\int_\Omega \prod_{j=1}^d \abs{\sum_{k_j\in L_j} \alpha_{j,k_j} u_{j,k_j}}^{\theta_jq} \dmu \leq \left( \int_\Omega \prod_{j=1}^d \abs{\sum_{k_j\in L_j} \alpha_{j,k_j} u_{j,k_j}}^{\theta_j} \dmu \right)^q \leq \prod_{j=1}^d \left( \sum_{k_j\in L_j} \abs{\alpha_{j,k_j}}\right)^{\theta_jq}
$$
   for all families $\{\alpha_{j,k_j}\}_{k_j\in L_j}$ of reals. Furthermore, by assumption, each indexing set $L_j$ is finite and each family $\{u_{j,k_j}\}_{k_j\in L_j}$ of functions is saturating. Therefore, by Proposition \ref{prop:functional_finite}, there exists $(\phi_j)$ satisfying the inequalities \eqref{eq:defa} and \eqref{eq:defb}.
   
   \item
   {\it Finite intersection property.} 
This follows from monotonicity and non-emptiness, together with preservation of finiteness and saturation under finite unions,
 $$
    \bigcap_{n=1}^N \Phi_{q_n,\{L_{j,n}\}}\supseteq \Phi_{\max\{q_1,\ldots,q_N\},\{\bigcup_{n=1}^N L_{j,n} \}}\neq \emptyset.
    $$

\end{itemize}
    The proof of the claim is completed.
  \end{proof} 
 We summarise what we have already established and what needs to be done. We need to prove that the set $\Phi_{1,\{K_j\}}$ is non-empty. By the set limit property, we know that
$$
\Phi_{1,\{K_j\}}=\bigcap_{q\in(0,1), L_j\in \cl_j} \Phi_{q,\{L_j\}}.
$$

Since the collection $\{\Phi_{q,\{L_j\}}\}_{q\in(0,1), L_j\in \cl_j}$ has the finite intersection property, we may conclude that its intersection $\Phi_{1,\{K_j\}}$ is non-empty if we can find some compact topological space in which all of the sets $\Phi_{q,\{L_j\}}$ are closed.

Now we use the raising-to-a-power approach described in the introduction: 
we will work with the powers $(\psi_j):=(\phi^{1/p}_j)$ (for any $1< p < \infty $) instead of functions $(\phi_j)$ themselves. Indeed, by the defining inequality \eqref{eq:defb} together with the facts that $L_j\ni k_{j,0}$ and $w_j:=u_{j,k_{j,0}}$, we have
$$
\norm{(\psi_j)}_{L^p(w_1)\times \cdots \times L^p(w_d)}:=\max_{j=1,\ldots,d} \int \abs{\psi_j}^p w_j\dmu \leq 1,$$
whereas
$$
\norm{(\phi_j)}_{L^1(w_1)\times \cdots \times L^1(w_d)}:=\max_{j=1,\ldots,d} \int \abs{\phi_j} w_j \dmu\leq 1.
$$
In other words, the $d$-tuple $(\psi_j)$ of powers of functions is contained in the unit ball of $L^p(w_1)\times\cdots\times L^p(w_d)$, which, by the Banach--Alaoglu theorem, is weakly compact since we are in the reflexive range $1 < p < \infty$; whereas the $d$-tuple $(\phi_j)$ of functions themselves is in the (non-compact) unit ball of $L^1(w_1)\times\cdots\times L^1(w_d)$. 

Therefore, we define the set $\Psi_{q,\{L_j\}}^p$ to be the component-wise $1/p$-th powers of $\Phi_{q,\{L_j\}}$, that is
$$ \Psi_{q,\{L_j\}}^p := \{(\phi_1^{1/p}, \dots , \phi_d^{1/p}) \, : \, (\phi_1, \dots , \phi_d) \in \Phi_{q,\{L_j\}}\},$$ 
and it suffices to show that each of these sets is weakly closed in the unit ball $B_{L^p(w_1)\times\cdots\times L^p(w_d)}$ of the normed product space $L^p(w_1)\times\cdots\times L^p(w_d)$. We do this by showing that 
$ \Psi_{q,\{L_j\}}^p$ is convex and norm-closed, and then appealing to the Hahn--Banach theorem.

We first note that from the defining inequalities \eqref{eq:defa} and \eqref{eq:defb}, together with the fact that the functions $\br^d_+\ni (a_j) \mapsto a_i^p \in \br_+$ and $\br^d_+\ni (a_j) \mapsto \left(\prod_{j=1}^d a_j^{\theta_j}\right)^{pq'}$ are convex, it follows that the sets $\Psi_{q,\{L_j\}}^{p}$ are convex.

Secondly, we observe that the sets $\Psi_{q,\{L_j\}}^{p}$ are norm-closed.
Indeed, assume that $\psi_n:=(\psi_{j,n})$ converges to $\psi:=(\psi_j)$ in the $L^p(w_1)\times\cdots\times L^p(w_d)$ norm. Thus, each component $\psi_{j,n}$ converges to $\psi_j$ in the $L^p(w_j)$ norm. By choosing a subsequence $\psi_{n_1(n)}$ of $\psi_n$ such that the first component $\psi_{n_1(n),1}$ converges to $\psi_{1}$ $\mu$-almost everywhere, a further subsequence $\psi_{n_2(n)}$ of $\psi_{n_1(n)}$  such that the second component $\psi_{n_2(n),2}$ converges to $\psi_{2}$ $\mu$-almost everywhere and so on, we can find a subsequence $\psi_{n_d(n)}$ of $\psi_n$ such that every component  $\psi_{n_d(n),j}$ converges to $\psi_{j}$ $\mu$-almost everywhere. Thus, by Fatou's lemma together with the pointwise (almost everywhere) convergence, the limit $\psi$ satisfies the inequalities defining the set $\Psi_{q,\{L_j\}}^{p}$.

Thirdly, it is a simple consequence of the Hahn--Banach theorem, often referred to as Mazur's theorem, that in any normed space, a convex norm-closed set is also weakly closed, see for example \cite[Theorem~3.12]{rudin}. Therefore, by the previous two observations,
$\Psi_{q,\{L_j\}}^{p}$ is weakly closed and we are done.

The proof of Theorem 2.1 is thus completed.

\medskip
{\em Final Remark.} Even when the $X_j$ are finite-dimensional normed lattices, we still need to run the arguments of this last section in order to obtain the case corresponding to $q=1$ of Proposition~\ref{proposition:maurey}, unless we happen to have extremisers $f_j$ such that $|T_jf_j| >0$ a.e. for each $j$.

\bibliographystyle{plain}

\bibliography{geometricmean}
\end{document}